\documentclass{amsart}
\usepackage{cite}
\usepackage{amssymb}
\usepackage{color} 
\usepackage{mathtools}
\definecolor{mcgreen}{RGB}{0,128,0} 
\definecolor{mcpurple}{RGB}{70,0,200} 
\definecolor{mcblue}{RGB}{0,0,150} 
\definecolor{mcgreen}{RGB}{0,128,0}
\definecolor{mcred}{RGB}{200,25,100}

\definecolor{mcbluegrey}{RGB}{225,230,240}
\definecolor{mcgreengrey}{RGB}{230,240,230}
\definecolor{mcmagenta}{RGB}{138,43,226}

\usepackage{mathrsfs}
\usepackage{amsmath}
\usepackage{thmtools, thm-restate}
\usepackage{verbatim}
\usepackage{hyperref} 
    \usepackage{cleveref}
\usepackage{memhfixc} 
\hypersetup{ 
    colorlinks=true,
    linkcolor=blue,    
    citecolor=blue,
    urlcolor=blue,
}
\usepackage{faktor}
\usepackage{float}
\usepackage[section]{placeins} 

\usepackage{tikz}
\usetikzlibrary{matrix}
\usetikzlibrary{calc}
\usepgflibrary{shapes}
\tikzset{diagram/.style={matrix of math nodes, inner sep=0pt, row
    sep=#1, column sep=2.5em, text height=1.5ex, text depth=.25ex,
    nodes={inner sep=1ex}}}
\tikzset{diagram/.default=2.5em}
\usetikzlibrary{arrows.meta}
\newcommand\diagram{\path node[diagram]}

\usepackage{amsthm}
\usepackage{thmtools, thm-restate}
\newtheorem{thm}{Theorem}[section]

\newtheorem{lem}[thm]{Lemma}

\newtheorem{defn}[thm]{Definition}

\newtheorem{rmk}[thm]{Remark}
\newtheorem{conj}[thm]{Conjecture}
 
\newtheorem{innercustomthm}{Theorem}
\newenvironment{customthm}[1]
  {\renewcommand\theinnercustomthm{#1}\innercustomthm}
  {\endinnercustomthm}
 
\newtheorem{innercustomcon}{Conjecture}
\newenvironment{customcon}[1]
  {\renewcommand\theinnercustomcon{#1}\innercustomcon}
  {\endinnercustomcon}

\newcommand{\ZZ}{\mathbb{Z}} 
\newcommand{\QQ}{\mathbb{Q}}

\newcommand{\PP}{\mathscr{P}}

\newcommand{\FF}{\mathbb{F}}
\newcommand{\OK}{\mathcal{O}_K}

\newcommand{\pp}{\mathfrak{p}} 
 
\newcommand{\bb}{\mathfrak{b}}

\newcommand{\aaa}{\mathfrak{a}} 
\newcommand{\mm}{\mathfrak{m}}

\newcommand{\ism}{\cong}

\newcommand{\inv}{^{-1}}

\newcommand{\defeq }{\vcentcolon=}
\newcommand{\Stab}{\operatorname{Stab}}

\newcommand{\spin}{\operatorname{spin}}
\newcommand{\Norm}{\operatorname{Norm}}
\newcommand{\ord}{\operatorname{ord}}
\newcommand{\Gal}{\operatorname{Gal}}

\newcommand{\Art}{\operatorname{Art}}

\newcommand{\preArt}[2]{\mathcal{A}_{#1}^{#2}}
\newcommand{\nRCF}[1]{\operatorname{nR}^{#1}}
\newcommand{\nRCG}[1]{{\operatorname{n\mathcal{C}l} }^{#1}}  
\newcommand{\totpos}[1]{{#1}\succ 0}

\newcommand{\switch}[1]{}

\title {Erratum to: On the Asymptotics of a Prime Spin Relation}
\author {Christine McMeekin}
\date{September 11, 2020}

\begin{document}
\maketitle
\begin{center}
      \textit{Original version included. See page \pageref{sec:original}}\\
\end{center}
\vspace{0.9cm}

In the case when the rational primes $p$ are assumed to split completely in $K/\QQ$, there is no change. We make corrections to the case when $p$ is inert. This does not significantly change the strength of the results because the inert case is uninteresting since $\spin(\pp,\sigma)=0$ for all $\sigma\in\Gal(K/\QQ)$ when $\pp$ is a prime of $K$ inert in $K/\QQ$. 

There are three variations of ``the spin relation" in question when $p$ is inert in $K/\QQ$ and so we consider the following three variations of the set $B$ from Definition \ref{defn:RB}. For $i=1,2,3$, the set $B_i$ below is a set of rational primes satisfying a spin relation. 
Throughout the paper, these sets are used interchangeably as $B$, and while the restriction of these sets to primes that split completely coincide, they differ in the inert case.
\begin{align*}
    B_1 \defeq \{& p\in\PP_\QQ^{2\ell}: \spin(\pp,\sigma)\spin(\pp,\sigma\inv)=1 \text{ for all nontrivial }\sigma\in\Gal(K/\QQ)\\
    &\text{where }\pp\text{ is a prime of }K\text{ above }p\}, \\
    B_2 \defeq \{&p\in\PP_\QQ^{2\ell}: (\alpha,\alpha^\sigma)_2=1 \text{ for all nontrivial }\sigma\in\Gal(K/\QQ)\text{ where }\alpha\text{ is a totally}\\
    &\text{positive generator of }\pp^h\text{ and }\pp\text{ is a prime of }K\text{ above }p\}, \quad \text{and} \\
    B_3 \defeq \{&p\in\PP_\QQ^{2\ell}:\spin(\pp,\sigma)=\spin(\pp,\sigma\inv)\text{ for all nontrivial }\sigma\in\Gal(K/\QQ)\\
    &\text{where }\pp\text{ is a prime of }K\text{ above }p\}. 
\end{align*}

We now state an extension and clarification of Theorem \ref{thm:littledensity:ext}.

\begin{customthm}{1.2*}\label{thm:littledensity:extCORRECTED}
Let $K:=K(n,\ell)$. 
The density of rational primes $p$ that satisfy the spin relations given by $p\in B_1,B_2,B_3$ are as follows.
\begin{align*}
    d(B_1)=&\frac{m_Kn+1}{n2^n} \\
    d(B_2)=&\frac{ 2^{n-1}(n-1)+ m_Kn+1}{n2^n} \\
    d(B_3)=&\frac{2^n(n-1)+m_Kn+1}{n2^n}
\end{align*}
\end{customthm}

Restricted to primes that split completely, $B_1\cap S=B_2\cap S=B_3\cap S$. However in the inert case,  $\spin(\pp,\sigma)=0$ for any $\sigma\in\Gal(K/\QQ)$ where $\pp$ is a prime of $K$. Therefore $B_1\cap I= \emptyset$ and $B_3\cap I=I$, whereas $d(B_2\cap I|I)=1/2$ by Lemma \ref{equidistribution:ext}, which remains valid for $B=B_2$. In equation $(5)$ of the proof of Theorem \ref{thm:littledensity:ext}, replacing $d(B\cap I|I)$ with $d(B_1\cap I|I)=0$, $d(B_2\cap I|I)=1/2$, or $d(B_3\cap I|I)=1$ yields the analogous statements given above.

We now state a modification of Conjecture \ref{conj:constantspindensity}

\begin{customcon}{1.1*}\label{conj:constantspindensityCORRECTED}
Fix $K:=K(n,\ell)$. Let $F$ denote the set of rational primes $p$ such that $\spin(\pp,\sigma)=1$ for all non-trivial $\sigma\in\Gal(K/\QQ)$ where $\pp$ is any prime of $K$ above $p$. Then
\[
d(F):=\frac{ m_Kn +1 }{n\left(\sqrt{2}\right)^{3n-1}} \quad \text{and} \quad d(F|S):= \frac{ m_Kn + 1 }{\left(\sqrt{2}\right)^{3n-1}}.
\]
\end{customcon}

Nothing has changed in the restriction to primes that split completely. In the general case, one can consider $d(F)=d(F|B_1)d(B_1|S)d(S)$. The reason $B_1$ is used is that the corresponding densities $d(F'|B_i')$ in $K$ for $i=1,2,3$ are more naturally accessible, and $d(F|B_1)=d(F'|B_1')$ since $B_1\subseteq S$. 

To justify the conjecture that $d(F'|B_i')=(1/2)^{(n-1)/2}$ for any $i=1,2,3$,
we first fix a generator $\tau$ of $\Gal(K/\QQ)$. Let $H\defeq\{\sigma=\tau^i:1\leq i\leq(n-1)/2\}$. Then $H$ is a maximal subset of $\Gal(K/\QQ)$ such that $\sigma\inv \notin H$ whenever $\sigma\in H$. Let $F_H'\defeq\{\pp\in\PP_K:\spin(\pp,\sigma)=1\text{ for all }\sigma\in H\}$. It does not make sense to consider the rational analogue of $F_H$ because the condition defining $H$ is not a statement about all elements of the Galois group, so choosing a different prime of $K$ above a fixed rational prime may yield different answers as to whether the prime belongs in $F_H$. Noting that $F'=F_H'\cap B_i'$ for any $i=1,2,3$ and that $B_i'$ is defined by a congruence condition, it is natural to conjecture based on Theorems 1.1 and 1.2 in \cite{FIMR} that $d(F'|B_i')=(1/2)^{(n-1)/2}$ provided $\spin(\pp,\sigma_1)$ and $\spin(\pp,\sigma_2)$ are independent for distinct $\sigma_1,\sigma_2\in H$.

We now state an extension and clarification of Theorem \ref{thm:Dstarbound2}.

\begin{customthm}{1.3*}\label{thm:Dstarbound2CORRECTED}
Let $K:=K(n,\ell)$. Then
\begin{align*}
    0<\frac{1}{n2^n}\leq &d(B_1)\leq\frac{1}{2n} \\
    0<\frac{2^{n-1}(n-1)+1}{n2^n}\leq &d(B_2) \leq\frac{1}{2} \\
    0<\frac{2^n(n-1)+1}{n2^n}\leq &d(B_3)\leq \frac{2n-1}{2n}
\end{align*}
\end{customthm}

For each $i=1,2,3$, $d(B_i)=d(B_i\cap I|I)d(I)+d(B_i\cap S|S)d(S)$. By Theorem \ref{thm:Dstarbound}, $d(B_i\cap S|S)\leq 1/2$ for each $i=1,2,3$. Since $d(B_1\cap I|I)=0$, $d(B_2\cap I|I)=1/2$, and $d(B_3\cap I|I)=1$, the upper bounds follow. The lower bounds follow trivially from Theorem \ref{thm:littledensity:ext}.

In addition, we make the following remarks.
\begin{enumerate}
    \item From Definition \ref{defn:starp}, $\star(p)=1$ for a rational prime $p$ if and only if $p\in B_2$, and $\star(\pp)=1$ for $\pp$ a prime of $K$ if and only if $\pp\in B_2'$. All of the proofs use the spin relation defined by $B_2$.
    \item The lines following Theorem \ref{StarMG} and Definition \ref{defn:RB} claim that $B_2=B_3$, which is only true when restricted to $S$.
    \item Lemma \ref{spindep} only holds if $\pp$ and $\pp^\sigma$ are coprime, affecting only the inert case.
    \item In the proofs of Lemma \ref{equidistribution} and Lemma \ref{equidistribution:ext}, in the application of Chebotarev's theorem, the extension $L/K$ is not cyclic as stated. Nevertheless it is abelian and Chebotarev's theorem yields the stated result.
    \item The proof of Lemma \ref{homMK:surj} is made more rigorous in \cite[4.6]{DRPjoint}.\\
\end{enumerate}

\textit{
We now include the original version for reference.
}

\newpage

\pagenumbering{arabic}

\section*{Original Version: On the Asymptotics of a Prime Spin Relation}\label{sec:original}




\begin{abstract}
For certain cyclic totally real number fields, we give formulas for the density of primes that satisfy a given spin relation. 
\end{abstract}

\tableofcontents

\section{Introduction}\label{sec:Introduction}

Let $K=K(n,\ell)$ denote a totally real number field that is cyclic over $\QQ$ with odd prime degree $n$ such that the class number of $K$ is odd, 2 is inert, and every totally positive unit is a square. Let $\ell$ denote the conductor of $K$ and let $\OK$ denote the ring of integers of $K$.

Let $\sigma \in \Gal(K/\QQ)$, $\sigma\neq 1$. Given an odd principal ideal $\aaa$, following \cite{FIMR}, we define the \textit{spin} of $\aaa$ (with respect to $\sigma$) to be
\[
\text{spin}(\aaa,\sigma) = \left(  \frac{\alpha}{ \aaa^\sigma}   \right)
\]
where $\aaa=(\alpha)$, $\alpha$ is totally positive, and $\left( \frac{\cdot}{\cdot} \right)$ denotes the quadratic residue symbol in $K$. 


The main results of this paper give a formula for the density of rational primes that exhibit the spin relation
\[
\spin(\pp,\sigma) = \spin(\pp,\sigma\inv) \quad \text{for all } \sigma\neq 1 \in \Gal(K/\QQ)
\]
where $\pp$ is a prime of $K$ above $p$. The formula is given in terms of  $n=[K:\QQ]$ and $m_K$, a computable and bounded invariant of the number field $K$.  Define 
\[
\mathbf{M}_4:= (\OK/4\OK)^\times / \left( (\OK/4\OK)^\times \right)^2
\]
where $\OK$ denotes the ring of integers of $K$.
We define \textit{the Starlight invariant} of the number field $K$ (denoted $m_K$) to be the number of non-trivial $\Gal(K/\QQ)$-orbits of $\mathbf{M}_4$ with representative $\alpha\in\OK$ such that the Hilbert symbol $(\alpha,\alpha^\sigma)_2=1$ for all non-trivial $\sigma\in \Gal(K/\QQ)$.

The main results of this paper are motivated by the following conjecture, which gives a computable and bounded formula for the density of rational primes with constant spin equal to $1$.

\begin{conj}\label{conj:constantspindensity} \footnote{See Conjecture \ref{conj:constantspindensityCORRECTED} for correction.}
Fix $K:=K(n,\ell)$. The density of rational primes $p$ such that $\spin(\pp,\sigma)=1$ for all non-trivial $\sigma\in\Gal(K/\QQ)$ is given by
\[
C_K:=\frac{ 2^{n-1}(n-1) + m_Kn +1 }{\left(\sqrt{2}\right)^{3n-1}n}.
\]
Restricting to rational primes that split completely in $K/\QQ$, the corresponding conditional density is given by
\[
C_{K,S}:= \frac{ m_Kn + 1 }{\left(\sqrt{2}\right)^{3n-1}}.
\]
\end{conj}


The reasoning behind Conjecture \ref{conj:constantspindensity} is as follows. We break up the density $C_K$ into a product of two densities as though one might break up a probability into a product of conditional probabilities. Then $C_K$ is the product of two densities; the first is $D_K$, the density of rational primes satisfying the  spin relation given in Theorem \ref{thm:littledensity:ext}, and the second is the conditional density of rational primes $p$ with $\spin(\pp,\sigma)=1$ for all non-trivial $\sigma\in\Gal(K/\QQ)$, assuming that $p$ satisfies the spin relation. 
Conjecture \ref{conj:constantspindensity} then asserts that
\[
C_K = D_K \left( \frac{1}{2} \right)^{\left( \frac{n-1}{2} \right)}.
\]
Note that the condition that $p$ satisfies the spin relation is a Cebotarev condition so by Theorems 1.1 and 1.2 in \cite{FIMR}, if $\spin(\pp,\sigma)$ and $\spin(\pp,\tau)$ are independent for $\sigma,\tau\in\Gal(K/\QQ)$ with $\sigma\neq \tau,\tau\inv$, then we arrive at Conjecture \ref{conj:constantspindensity}.

A corollary of Conjecture \ref{conj:constantspindensity} is a family of number fields $\{F_K(p)\}_p$ depending on $p$ such that $p$ is always ramified in $F_K(p)/\QQ$ and the density of rational primes that split as completely as possible in $F_K(p)/\QQ$ (given the ramification) is 
\[
\frac{C_{K,S}}{n}.
\]

\begin{restatable}{thm}{littledensityext}
\label{thm:littledensity:ext}
\footnote{See Theorem \ref{thm:littledensity:extCORRECTED} for correction.}
Let $K:=K(n,\ell)$. 
The density of rational primes $p$ that satisfy the spin relation
\[
\spin(\pp,\sigma) = \spin(\pp,\sigma\inv) \quad \text{for all } \sigma\neq 1 \in \Gal(K/\QQ)
\]
where $\pp$ is a prime of $K$ above $p$ is given by
\[
D_K=\frac{ 2^{n-1}(n-1)+ m_Kn+1}{2^nn}.
\] 
\end{restatable}

\begin{restatable}{thm}{Dstarboundtwo}
\label{thm:Dstarbound2}
\footnote{See Theorem \ref{thm:Dstarbound2CORRECTED} for correction.}
Let $K:=K(n,\ell)$. Then
\[
0<\frac{ 2^{n-1}(n-1)+1}{2^nn}\leq  D_K \leq \frac{1}{2}.
\]
\end{restatable}

Table \ref{tab:expectedmK} gives examples of computed Starlight invariants for cyclic number fields of degree $n$ over $\QQ$ and conductor $\ell$ for the given $n$ and $\ell$ values and it gives the corresponding density of primes satisfying the spin relation.
These values of $m_K$ were computed using magma \cite{magma}; the code can be found in Appendix $B$ of \cite{DRP}.

\begin{table}[H]
    \centering
    \bgroup
\def\arraystretch{1.5}
\setlength\tabcolsep{.4cm}
    \begin{tabular}{ |c || c| c| c| c| c| c| c|}\hline
$n$ & 3 & 5 & 7  & 11 & 13 & 17 & 19\\  \hline
$\ell$ & 7 &  11 & 43 & 23 & 53 & 103 & 191 \\ \hline
$m_K$ & 1 & 1 &  3 & 3 &  5 & 17 & 27   \\  \hline
$D_K$ & $\frac{1}{2}$ & $\frac{7}{16}$ & $\frac{29}{64}$& $\frac{467}{1024}$& $\frac{1893}{4096}$& $\frac{30849}{65536}$& $\frac{124187}{262144}$\\ \hline
    \end{tabular}
    \egroup
    \caption{Computed Starlight invariants and densities of the prime spin relation using Theorem \ref{thm:littledensity:ext}.}
    \label{tab:expectedmK}
\end{table}

We remark that we can simplify the restrictions on $K$ in the cubic case. For $n=3$, the assumption that the class number of $K$ is odd is sufficient to imply that every totally positive unit is a square due to results of Armitage and Fr\"ohlich \cite{AF}. 

\begin{thm}\cite{AF}
Let $K$ be a cyclic cubic number field with odd class number. Then every totally positive unit is a square. 
\end{thm}

\begin{proof}
Let $U:= \OK^\times$ denote the group of units, $U_T$ the totally positive units, and $U^2$ the square units. 
Observe $U^2\subseteq U_T \subseteq U$. Then we have a surjective homomorphism
\[
\phi: \frac{U}{U^2} \to \frac{U}{U_T}.
\]
If none of the nontrivial class representatives of $U/U^2$ are totally positive then $\phi$ is injective.  By Theorem $V$ in \cite{AF}, all signatures are represented by units. Square units are always totally positive and there are 8 signatures and 8 classes of units mod squares, so each class of $U/U^2$ must have a different signature. Therefore $U_T$$=$$U^2$.
\end{proof}


\section{The Spin of Prime Ideals}\label{sec:spin}

Let $K:=K(n,\ell)$ and let $h(K)$ denote the class number of $K$.

\begin{defn}\cite{FIMR}\label{defn:spin} Let $\sigma\neq 1 \in \Gal(K/\QQ)$. Given an odd principal ideal $\aaa$, we define the \textit{spin} of $\aaa$ (with respect to $\sigma$) to be
\[
\text{spin}(\aaa,\sigma) = \left(  \frac{\alpha}{ \aaa^\sigma}   \right)
\]
where $\aaa^{h(K)}=(\alpha)$, $\alpha$ is totally positive, and $\left( \frac{\cdot}{\cdot} \right)$ denotes the quadratic residue symbol in $K$. 
\end{defn}

Spin is well-defined; since every totally positive unit is a square, the choice of totally positive generator $\alpha$ does not affect the quadratic residue and Lemma \ref{RCFh} asserts the existence of a totally positive generator.



Lemma 11.1 in \cite{FIMR} states that the product $\text{spin}(\pp,\sigma)\text{spin}(\pp,\sigma\inv)$ is a product of Hilbert symbols at places dividing 2. We restate this more explicitly in Lemma \ref{spindep}.

For a place $v$ of $K$, let $K_{(v)}$ denote the completion of $K$ at $v$. 
For $a,b\in K$ co-prime to $v$, the Hilbert Symbol is defined such that $(a,b)_v:=1$ if the equation $ax^2+by^2=z^2$ has a solution $x,y,z\in K_{(v)}$ where at least one of $x$, $y$, or $z$ is nonzero and $(a,b)_v:=-1$ otherwise.

\begin{lem}\cite{FIMR}\label{spindep}
Let $K:=K(n,\ell)$.
Let $\alpha$ be a totally positive generator of the odd prime ideal $\pp\subseteq\OK$. Then
\[
\text{spin}(\pp,\sigma)\text{spin}(\pp,\sigma\inv) = \prod_{v|2}(\alpha,\alpha^\sigma)_v.
\]
In particular, if $\alpha\equiv 1 \bmod 4$ then $\prod_{v|2}(\alpha,\alpha^\sigma)_v=1$.
Since 2 is inert in $K/\QQ$,
\[
\text{spin}(\pp,\sigma)\text{spin}(\pp,\sigma\inv) = (\alpha,\alpha^\sigma)_2.
\]
\end{lem}

\begin{proof}
See Lemma 11.1 in \cite{FIMR} or use the standard fact of Hilbert symbols that $\prod_{v}(\alpha,\alpha^\sigma)_v=1$.
\end{proof}


\section{Some Class Field Theory}\label{sec:narrowwide}

We now diverge momentarily from the spin of prime ideals to discuss some class field theory in the case when every totally positive unit is a square. 
 We say a modulus is \textit{narrow} whenever it is divisible by all infinite places. We say a modulus is \textit{wide} whenever it is not divisible by any infinite place. We say a ray class group or ray class field is narrow or wide whenever its defining modulus is narrow or wide respectively.
 Let $U:=\OK^\times$ denote the group of units of $K$, let $U_T$ denote the totally positive units, and let $U^2$ denote the square units.
The following lemma is an exercise in class field theory.
    
    \begin{lem}\label{RCFh} \switch{[RCFh]}
Let $K$ be a totally real number field. The following are equivalent.
\begin{enumerate}
    \item $U_T=U^2$.
    \item The \textit{narrow} and \textit{wide} Hilbert class groups of $K$ coincide.
    \item Every principal ideal of $K$ has a totally positive generator.
\end{enumerate}

\end{lem}

\begin{proof}
To show the equivalence of (1) and (2), apply Theorem V.1.7 in \cite{MilneCFT} using the modulus given by the product of all infinite places. Statements (3) and (2) are equivalent by the definitions of narrow and wide class groups.
\end{proof}


\begin{restatable}{defn}{MKdefn}\label{defn:MK} Let $K:=K(n,\ell)$.
For $q$ a power of 2, we define the group
 \[
\mathbf{M}_q:=  \left(\OK/q\OK\right)^\times/\left(\left(\OK/q\OK\right)^\times\right)^2.
\]

The Galois group $\Gal(K/\QQ)$ acts on $\mathbf{M}_q$ in the natural way.
\end{restatable} 

We will primarily be interested in $\mathbf{M}_4$. 
We will see in Lemma \ref{homMK:surj} that $\mathbf{M}_q$ is canonically isomorphic to a quotient of the narrow ray class group over $K$ of conductor $q$ modulo squares.  
    
    
\begin{lem}\cite{Mundy}\label{Mundy}\switch{[Mundy]} Let $K$ be a cyclic number field of odd 
degree $n$ over $\QQ$ such that $2$ is inert in $K$.  Then as vector spaces 
\[
\mathbf{M}_4\ism (\ZZ/2)^n.
\]
Furthermore, the invariants of the action of $\Gal(K/\QQ)$ are exactly $\pm 1\in \mathbf{M}_4$.
\end{lem}

\begin{proof} This proof is due to Sam Mundy \cite{Mundy}.
Consider the exact sequence
\begin{equation}\label{MundyExact1}
0 \to 1 + 2(\OK/4) \to (\OK/4)^\times \to (\OK/2)^\times \to 1.
\end{equation}
Note that $\OK/2 \ism \FF_{2^n}$ because $K$ is cyclic of odd degree and 2 is inert in $K$. Also, $G\ism \Gal(\FF_{2^n}/\FF_2)$.

Viewing $\FF_{2^n}$ as an additive group with Galois action by $G\ism \Gal(\FF_{2^n}/\FF_2)$, there is an isomorphism of Galois modules given by
\[
\psi: \FF_{2^n} \ism \OK/2 \to  1 + 2(\OK/4)
\]
\[
\psi: x \mapsto 1+2x.
\]
This map is easily seen to be a Galois equivariant  homomorphism. Injectivity and surjectivity follow from considering 2-adic expansions of elements in $\OK/4$. Since $\psi$ is an isomorphism we can rewrite the exact sequence of Galois modules in equation \ref{MundyExact1} as
\begin{equation}\label{MundyExact2}
0 \to \FF_{2^n} \to (\OK/4)^\times \to \FF_{2^n}^\times \to 1.
\end{equation}

Next consider the diagram of exact sequences below.
\[
\begin{tikzpicture}
\diagram (m)
{ 0 & \FF_{2^n} & (\OK/4)^\times & \FF_{2^n}^\times &  1 \\
0 &  \FF_{2^n} & (\OK/4)^\times & \FF_{2^n}^\times &  1 \\
};
\path [->] 
           (m-1-1) edge node [above] {} (m-1-2)
           (m-1-2) edge node [above] {} (m-1-3)
           (m-1-3) edge node [above] {} (m-1-4)
           (m-1-4) edge node [above] {} (m-1-5)
           
            (m-2-1) edge node [above] {} (m-2-2)
           (m-2-2) edge node [above] {} (m-2-3)
           (m-2-3) edge node [above] {} (m-2-4)
           (m-2-4) edge node [above] {} (m-2-5)
           
            (m-1-2) edge node [right] {$2(\cdot)$} (m-2-2)
           (m-1-3) edge node [right] {$(\cdot)^2$} (m-2-3)
           (m-1-4) edge node [right] {$(\cdot)^2$} (m-2-4)
          ;
\end{tikzpicture}
\]

The first vertical map is multiplication by 2, which is the zero map. The next two vertical maps are squaring. The third vertical map is an isomorphism because $\FF_{2^n}^\times$ is cyclic of odd order. Recall that
\[
\mathbf{M}_4:= \left( \OK/4\OK \right)^\times/\textrm{squares}.
\]
Then we apply the snake lemma to the diagram below.

\[
\begin{tikzpicture}
\diagram (m)
{ &                   &                          &           1 &  \\
0 & \FF_{2^n} & (\OK/4)^\times & \FF_{2^n}^\times &  1 \\
0 &  \FF_{2^n} & (\OK/4)^\times & \FF_{2^n}^\times &  1 \\
   &  \FF_{2^n} & \mathbf{M}_4 & 1 & \\
};
\path [->] 
	(m-1-4) edge node [above] {} (m-2-4)

           (m-2-1) edge node [above] {} (m-2-2)
           (m-2-2) edge node [above] {} (m-2-3)
           (m-2-3) edge node [above] {} (m-2-4)
           (m-2-4) edge node [above] {} (m-2-5)
           
            (m-3-1) edge node [above] {} (m-3-2)
           (m-3-2) edge node [above] {} (m-3-3)
           (m-3-3) edge node [above] {} (m-3-4)
           (m-3-4) edge node [above] {} (m-3-5)
           
           (m-4-2) edge node [above] {} (m-4-3)
           (m-4-3) edge node [above] {} (m-4-4)
           
            (m-2-2) edge node [right] {$0$} (m-3-2)
           (m-2-3) edge node [right] {$(\cdot)^2$} (m-3-3)
           (m-2-4) edge node [right] {$(\cdot)^2$} (m-3-4)
           
           (m-3-2) edge node [right] {} (m-4-2)
           (m-3-3) edge node [right] {} (m-4-3)
           (m-3-4) edge node [right] {} (m-4-4)
          ;
\end{tikzpicture}
\]

The snake lemma gives us the exact sequence of $G$-modules
\[
0 \to  \FF_{2^n} \to \mathbf{M}_4 \to 1.
\]
Therefore $\mathbf{M}_4\ism \FF_{2^n}$ as $G$-modules. The invariants of $\FF_{2^n}$ are $\FF_2$. Tracing through the isomorphism we see that this corresponds to the invariants $\{\pm 1\}$ in $\mathbf{M}_4$.
\end{proof}

    
    





Let $\mathbb{M}_{q,G}$ denote the set of $\Gal(K/\QQ)$-orbits of $\mathbf{M}_q$ for $q$ a power of 2.
Recall that we say a modulus of $K$ is narrow whenever $\mm_\infty$ divides the modulus where $\mm_\infty$ is the product of all infinite places of $K$. Letting $\mm$ denote a narrow modulus with finite part $\mm_0$, let
$J_K^\mm=J_K^{\mm_0}$ denote the group of fractional ideals of $K$ prime to $\mm_0$ and let $P_K^\mm=P_K^{\mm_0}$ denote the subgroup of $J_K^\mm$ formed by the principal 
ideals with generator $\alpha\in K^\times$ such that $\ord_2(q)\leq\ord_2(\alpha)$ and $\totpos{\alpha}$.
We let $\PP_K^\mm=\PP_K^{\mm_0}$ denote the set of prime ideals of $\OK$ co-prime to $\mm_0$ so that $J_K^\mm$ is generated by $\PP_K^\mm$.

\begin{defn}\label{defn:r} \switch{[defn:r]}
Let $K:=K(n,\ell)$. Let $q\geq 4$ be a power of 2.

\begin{enumerate}
    \item 
    Define the map
\begin{align*}
\mathbf{r}_0: & 
\PP_K^{2}\to  \mathbf{M}_q \\
 & \pp \mapsto \alpha
\end{align*}
where $\alpha\in\OK$ is a totally positive generator for the principal ideal $\pp^{\text{$h(K)$}}$. 
\item 
Define the map
\begin{align*}
\mathbf{r}: & 
\PP_\QQ^{2}\to \text{ $\mathbb{M}_{q,G}$ } \\
 & p \mapsto [\mathbf{r}_0(\pp)]
\end{align*}
where $\pp$ is any prime in $K$ above $p$. 
Here $[\alpha]$ denotes the $\Gal(K/\QQ)$-orbit of $\alpha\in \mathbf{M}_4$ considered in $\mathbb{M}_{q,G}$. 
\end{enumerate}
\end{defn}
The map $\mathbf{r}_0$ is well-defined out of 
$\PP_K^{2}$; recall that by Lemma \ref{RCFh}, $U_T=U^2$ is equivalent to the coincidence of the narrow and wide Hilbert class groups so $U_T=U^2$ if and only if all principal ideals have a totally positive generator. Since squares are trivial in $\mathbf{M}_q$ by definition and $U_T=U^2$, the map $\mathbf{r}_0$ is well-defined.

The map $\mathbf{r}$ is well-defined out of $\PP_\QQ^{2}$ because $\mathbb{M}_{q,G}$ is the quotient of $\mathbf{M}_q$ by the $\Gal(K/\QQ)$-action so different choices of primes $\pp$ of $K$ above $p$ give the same result; $\mathbf{r}_{0}(\pp^\sigma) = \mathbf{r}_{0}(\pp)^\sigma$ for $\sigma\in\Gal(K/\QQ)$ and $\pp$ an odd prime of $K$.

 Since $J_K^q$ is generated by $\PP_K^q=\PP_K^2$, the map $\mathbf{r}_0$ induces a homomorphism 
 \[
{\varphi}_0: J_K^q \to \mathbf{M}_q.
 \]

\begin{lem}\label{homMK}\label{homMK:surj}\label{rmk:MKcanonism}\switch{[rmk:MKcanonism]}\switch{[homMK:surj]}

Let $K:=K(n,\ell)$.
The homomorphism $\varphi_0:  J_K^q \to \mathbf{M}_q$ induces a canonical surjective homomorphism
\[
\varphi: \nRCG{q}_K \to  \mathbf{M}_q.
\]
\end{lem}

\begin{proof}
We first show the induced homomorphism is well-defined. By Proposition V.1.6 in \cite{MilneCFT}, every element of $\nRCG{q}$ is represented by an integral ideal. Let $\aaa$ and $\bb$ be two integral ideals representing the same element of $\nRCG{q}$. Then by Proposition V.1.6 in \cite{MilneCFT}, there exist nonzero $a,b\in\OK$ such that  
\begin{align*}
&b\aaa=a\bb, \\
&a \equiv b \equiv 1 \bmod q, \quad \text{and}\\
&\totpos{ab}.
\end{align*}

Since $\varphi_0: J_K^q \to \mathbf{M}_q$ is a homomorphism, 
\[
\varphi_0(b\OK)\varphi_0(\aaa) = \varphi_0(a\OK) \varphi_0(\bb).
\]

Noting that $h(K)$ is odd and squares are trivial in $\mathbf{M}_q$ by definition, $\varphi_0$ maps any principal integral ideal $(\alpha)$ to the class in $\mathbf{M}_q$ containing the representative $\alpha\in\OK$ where $\alpha$ is a totally positive generator.


Since $U_T=U^2$, 
 every principal ideal of $\OK$ has a totally positive generator so there exists a unit $u\in\OK^\times$ such that $\totpos{ua}$ and  $\varphi_0(a) = ua$. Since $\totpos{ab}$, then $\totpos{u\inv b}$ so $\varphi_0(b) = u\inv b$. We know that $a \equiv b \equiv 1 \bmod q$. Since squares are trivial in $\mathbf{M}_q$ by the definition of $\mathbf{M}_q$, this implies 
\begin{align*}
&u^2 a \equiv b \quad \text{ in } \mathbf{M}_q\\
\implies & ua \equiv u\inv b \quad \text{ in } \mathbf{M}_q\\
\implies & \varphi_0(a\OK)=\varphi_0(b\OK)\\
\implies & \varphi_0(\aaa) = \varphi_0(\bb).
\end{align*}
Therefore the homomorphism $\varphi_0$ induces a well-defined homomorphism from $\nRCG{q}_K$.

We now show the homomorphism is a canonical surjective homomorphism.
Let $\mm$ be the narrow modulus with finite part $q$. Let 
\begin{align*}
&K_\mm:=\{a\in K^\times: \ord_2(a)=0\},\\
&K_{\mm,1}:=\{a \in K^\times: \ord_2(a-1)\geq \ord_2(q), \totpos{a}\},\\
&U_{\mm,1}:=K_{\mm,1}\cap U.
\end{align*}
 Let $X\in \mathbf{M}_q$. Consider the exact sequence from 
Theorem V.1.7 in \cite{MilneCFT};
\[
1 \to U/U_{\mm,1} \to K_\mm/K_{\mm,1} \to \nRCG{\mm}_K \to C \to 1
\]
and the canonical isomorphism 
\begin{equation}\label{eqn:canonismKmm1}
K_\mm/K_{\mm,1} \ism \left(\pm\right)^n \times \left( \OK/q \right)^\times.
\end{equation}

Consider only the 2-part of each group. Then since $h(K)$ is odd, we have the short exact sequence
\[
1 \to (U/U_{\mm,1})[2^\infty] \to (K_\mm/K_{\mm,1})[2^\infty] \to (\nRCG{\mm}_K)[2^\infty] \to 1.
\]

Note that since squaring sends all signatures to the trivial signature, the canonical isomorphism in equation \ref{eqn:canonismKmm1} induces a canonical isomorphism on the 2-part modulo squares;
\[
(K_\mm/K_{\mm,1})[2^\infty]/(K_\mm/K_{\mm,1})[2^\infty]^2 \ism (\pm)^n \times \mathbf{M}_q.
\]

 Consider the squaring map and apply the snake lemma to get the following commutative diagram of exact sequences;
\begin{center}
\begin{tikzpicture}
\diagram (m)
{ 1 & (U/U_{\mm,1})[2^\infty] & (K_\mm/K_{\mm,1})[2^\infty] & (\nRCG{q}_K)[2^\infty] &  1  \\
 1 & (U/U_{\mm,1})[2^\infty] & (K_\mm/K_{\mm,1})[2^\infty] & (\nRCG{q}_K)[2^\infty] &  1  \\
  & U/U^2 & (\pm)^n \times \mathbf{M}_q   & \nRCG{q}_K/(\nRCG{q}_K)^2 & 1 &   \\
   & 1 & 1 & 1 &  &   \\};
\path [->] 
           (m-1-1) edge node [above] {} (m-1-2)
           (m-1-2) edge node [above] {} (m-1-3)
           (m-1-3) edge node [above] {} (m-1-4)
           (m-1-4) edge node [above] {} (m-1-5)
           (m-2-1) edge node [above] {} (m-2-2)
           (m-2-2) edge node [above] {} (m-2-3)
           (m-2-3) edge node [above] {} (m-2-4)
           (m-2-4) edge node [above] {} (m-2-5)
           
           (m-3-2) edge node [above] {} (m-3-3)
           (m-3-3) edge node [above] {$\psi$} (m-3-4)
           (m-3-4) edge node [above] {} (m-3-5)
           ;
\path [->] 
           (m-1-2) edge node [right] {$(\cdot)^2$} (m-2-2)
           (m-1-3) edge node [right] {$(\cdot)^2$} (m-2-3)
           (m-1-4) edge node [right] {$(\cdot)^2$} (m-2-4)
           
           (m-2-2) edge node [right] {} (m-3-2)
           (m-2-3) edge node [right] {} (m-3-3)
           (m-2-4) edge node [right] {} (m-3-4)
           
           (m-3-2) edge node [right] {} (m-4-2)
           (m-3-3) edge node [right] {} (m-4-3)
           (m-3-4) edge node [right] {} (m-4-4)
          ;
\end{tikzpicture}
\end{center}

Then $\psi$ induces an isomorphism
\[
\psi: \faktor{((\pm)^n \times \mathbf{M}_q)}{\operatorname{image}(U/U^2)}  \longrightarrow \faktor{\nRCG{q}_K}{(\nRCG{q}_K)^2}.
\]
Tracing through the definitions of the maps,  $\varphi\circ \psi$ is surjective (it is essentially the identity). Therefore $\varphi$ is surjective.


\end{proof}


\section{An Equidistribution Lemma}

\begin{defn}\label{defn:density} Let $S$ be a set of primes and let $R\subseteq S$. If the limit exists, then the \textit{restricted density} of $R$ (restricted to $S$) is defined as
\[
d(R|S):= \lim_{N\to \infty} \frac{\# R_N}{\#S_N}\]
where $S_N$ and $R_N$ denote the set of primes in $S$ and $R$ respectively of absolute norm less than $N\in \ZZ_+$. 
\end{defn}

Recall that $\PP_\QQ^{m}$ denotes the set of rational primes not dividing $m$ and  $\PP_K^{\mm}$ denotes the set of primes of $K$ not dividing $\mm$.
Letting $\pp$ be a prime of $K$ above a rational prime $p$, denote the corresponding inertia degree $f_{K/\QQ}(p)=f_{K/\QQ}(\pp)$ (well-defined because $K$ is Galois over $\QQ$). That is, 
\[
f_{K/\QQ}(p) = f_{K/\QQ}(\pp) = \frac{\#D}{\#E}
\]
where $D$ is the decomposition group of $\pp$ for the extension $K/\QQ$ and $E$ is the inertia group. 

\begin{defn}\label{defn:SI}
Let $K=K(n,\ell)$. Define the following sets of rational primes. 
\[
S:=\{p\in\PP_\QQ^{2\ell}: f_{K/\QQ}(p) = 1\} \quad \text{and} \quad
I:=\{p\in\PP_\QQ^{2\ell}: f_{K/\QQ}(p) = n\}. 
\]
Define the following sets of primes of $K$.
\[
S':=\{\pp\in\PP_K^{2\ell}: f_{K/\QQ}(\pp) = 1\} \quad \text{and} \quad
I':=\{\pp\in\PP_K^{2\ell}: f_{K/\QQ}(\pp) = n\}. 
\]
\end{defn}

That is, $S\subseteq \PP_\QQ^{2\ell}$ is the set of odd rational primes that split completely in $K/\QQ$ and $I\subseteq \PP_\QQ^{2\ell}$ is the set of odd rational primes that are inert in $K/\QQ$. 
Furthermore, $S'$ is the set of primes of $K$ laying above the primes in $S$ and $I'$ is the set of primes of $K$ laying above the  primes in $I$.

Since $K/\QQ$ is cyclic of prime degree $n$, then $f_{K/\QQ}(p) = 1$ or $n$ for all $p\in\PP_\QQ^{2\ell}$ so in this case, $\PP_\QQ^{2\ell}$ is the disjoint union of $S$ and $I$.
The next Lemma asserts that for $K:=K(n,\ell)$, the primes are equidistributed in $\mathbf{M}_4$ via the map $\mathbf{r}_0$. 

Although the equidistribution generalizes to $\mathbf{M}_q$, note that the number of elements of $\mathbf{M}_8$ for example is different than the number of elements of $\mathbf{M}_4$ so the generalized statement would need to be adjusted accordingly.


\begin{lem}\label{equidistribution}\switch{[equidistribution]} Let $K:=K(n,\ell)$. 

\begin{enumerate}
    \item\label{equidistribution:allp}  For any $\alpha\in \mathbf{M}_4$, the density of $\pp\in \PP_K^{2\ell}$ such that $\varphi(\pp)=\alpha$ is $\frac{1}{2^n}$. That is, 
    \[
    d(\mathbf{r}_0\inv(\alpha)| \PP_K^{2\ell}) = \frac{1}{\#\mathbf{M}_4} = \frac{1}{2^n}.
    \]
    \item\label{equidistribution:spcomp}  Furthermore, the density does not change when we restrict to primes of $K$ that split completely in $K/\QQ$. That is, 
    \[
    d(\mathbf{r}_0\inv(\alpha) \cap S' | S') = \frac{1}{\#\mathbf{M}_4} = \frac{1}{2^n}.
    \]
\end{enumerate}

\end{lem}

\begin{proof} 
We first prove part (1).
Recall that $\nRCF{4}=\nRCF{4}_K$ denotes the narrow ray class field over $K$ of conductor $4\mm_\infty$.
Let $G:=\Gal(\nRCF{4}/K)$. Define $H\leq G$ to be
\[
H:= \Art(\ker(\varphi))
\]
where $\Art$ denotes the Artin isomorphism. In other words, we define $H$ by  the following commutative diagram of exact sequences
\begin{center}
\begin{tikzpicture}
\diagram (m)
{ 1 & \ker(\varphi) & \nRCG{4} & \mathbf{M}_4 & 1  \\
  1 & H &  G &  \mathbf{M}_4 & 1 \\};
\path [->] 
           (m-1-1) edge node [above] {} (m-1-2)
           (m-1-2) edge node [above] {} (m-1-3)
           (m-1-3) edge node [above] {$\varphi$} (m-1-4)
           (m-1-4) edge node [above] {} (m-1-5)
           (m-2-1) edge node [above] {} (m-2-2)
           (m-2-2) edge node [above] {} (m-2-3)
           (m-2-3) edge node [above] {} (m-2-4)
           (m-2-4) edge node [above] {} (m-2-5)
           ;
\path [->] 
           (m-1-2) edge node [right] {$\Art$} (m-2-2)
           (m-1-3) edge node [right] {$\Art$} (m-2-3)
           (m-1-4) edge node [right] {\text{id.}} (m-2-4)
          ;
\end{tikzpicture}
\end{center}
where surjectivity of $\varphi$ is proven in Lemma \ref{homMK:surj}.
Let $L$ be the fixed field of $H$ so that $\Gal(L/K) \ism G/H$. 
\begin{center}
\begin{tikzpicture}
\diagram (m)
{ \nRCF{4} \\
   L \\
   K \\};
\path [-] 
 	(m-1-1) edge node [right]{} (m-2-1)
	(m-2-1) edge node [right]{} (m-3-1)
	(m-1-1) edge [bend left=60] node [right]{$G$} (m-3-1)
	(m-1-1) edge [bend right=60] node [left]{$H$} (m-2-1);
\end{tikzpicture}
\end{center}

This induces a canonical isomorphism
\[
\mathbf{M}_4\ism G/H \ism \Gal(L/K).
\]
Note for the proof of part (2) that the action from $\Gal(K/\QQ)$ commutes with the diagram so $L/\QQ$ is Galois with $\Gal(L/K) \unlhd \Gal(L/\QQ)$.

For a Galois extension $F/E$ of conductor dividing $\mm$, let
\begin{align*}
\preArt{F|E}{E}(\tau)&\defeq \{ p\in\PP_E^{\mm}: \Art_{F|E}(p) = \left<\tau\right>\},\\
\preArt{F|E}{F}(\tau)&\defeq \{ \pp\in\PP_F^{\mm}: \pp \text{ lies above }p\in\preArt{F|E}{E} \},
\end{align*}
where $\tau\in\Gal(F/E)$ and $\left<\tau\right>$ denotes the conjugacy class of $\tau$.

Let $\alpha\in \mathbf{M}_4$.
Let $\sigma\in\Gal(L/K)$ corresponding to $\alpha\in\mathbf{M}_4$. Then $\preArt{L|K}{K}(\sigma) = \mathbf{r}_0\inv(\alpha)$ is the set of primes of $K$ that map to $\alpha$ via $\mathbf{r}_0$

Theorem 4 in \cite{Serre} asserts Cebotarev's density theorem for natural density, (or see \cite{NeukirchANT} Theorem VII.13.4 for a simpler proof using Dirichlet density).
By the special case of Cebotarev's density theorem in which $L/K$ is cyclic, $\preArt{L|K}{K}(\sigma)$ has a density and it is given by 
\[
\frac{1}{\#\Gal(L/K)} = \frac{1}{\#\mathbf{M}_4}.
\]


The first asserted equality of part (1) is proved. The second equality of part (1) is true by Lemma \ref{Mundy}. 

To prove part (2), observe that 
\[
d(\mathbf{r}_0\inv(\alpha)| \PP_K^{2\ell}) =
d(\mathbf{r}_0\inv(\alpha) \cap S' | S') d(S'|\PP_K^{2\ell}) + d(\mathbf{r}_0\inv(\alpha) \cap I' | I') d(I'|\PP_K^{2\ell}).
\]
Since $d(S'|\PP_K^{2\ell}) =1$, $d(I'|\PP_K^{2\ell})=0$, and $0 \leq d(\mathbf{r}_0\inv(\alpha) \cap I' | I') \leq 1$,
\[
d(\mathbf{r}_0\inv(\alpha)| \PP_K^{2\ell}) =
d(\mathbf{r}_0\inv(\alpha) \cap S' | S').
\]

\end{proof}

\section{Property Star and the Starlight Invariant}

Let $K:=K(n,\ell)$.
Recall that $\mathbb{M}_{4,G}$ denotes the set of $\Gal(K/\QQ)$-orbits of $\mathbf{M}_4$ and recall the statement of Lemma \ref{spindep} from \cite{FIMR},
which motivates the following.

\begin{thm}
\label{StarMG}\switch{[StarMG]}
Let $K:=K(n,\ell)$. 
Let $\alpha\in\OK$ denote a 
representative of $[\alpha]\in$ $\mathbb{M}_{4,G}$. 
Define the map
\begin{align*}
\text{$\star$}: & \text{ $\mathbb{M}_{4,G}$ } \to \{\pm 1\} \\
 & [\alpha] \mapsto \left\{ 
  \begin{array}{l l}
    1 & \quad \text{if }  (\alpha,\alpha^\sigma)_2 = 1 \text{ for all non-trivial } \sigma\in \Gal(K/\QQ)  \\
    -1 & \quad \text{otherwise.}\\
  \end{array} \right.
\end{align*}
Then $\star$ is a well-defined map.
\end{thm}

Note that by Lemma \ref{spindep}, a rational prime $p$ satisfies the spin relation 
\[
\spin(\pp,\sigma) = \spin(\pp,\sigma\inv) \quad \text{for all } \sigma\neq 1 \in \Gal(K/\QQ),
\]
where $\pp$ is a prime of $K$ above $p$ exactly when $\star\circ\mathbf{r}(p)=1$ where $\mathbf{r}$ is as defined in Definition \ref{defn:r}.

\begin{proof}
We will show that $\star$ is well-defined out of $\mathbf{M}_4$. Then because $\star$ is a property of the full Galois orbit, $\star$ is well-defined out of $\mathbb{M}_{4,G}$.

Let $\alpha,\beta\in\OK$ be two 
representatives of the same class in $\mathbf{M}_4$ so
\[
\alpha \equiv \beta \gamma^2 \bmod 4\OK \quad \text{for some $\gamma\in\OK$.}
\]

If $\alpha \equiv \beta \gamma^2 \bmod 8\OK$ then we can apply Lemma 2.3 from \cite{FIMR} to see that $(\alpha,\alpha^\sigma)_2=(\beta,\beta^\sigma)_2$ for all $\sigma\in\Gal(K/\QQ)$. 
Therefore, we may assume 
\[
\alpha \equiv 5 \beta \gamma^2 \bmod 8\OK.
\]

Suppose $(\alpha,\alpha^\sigma)_2=1$. Then by Lemma 2.3 in \cite{FIMR}, since $\alpha \equiv 5 \beta \gamma^2 \bmod 8\OK$, 
\begin{align*}
    & \left(5 \beta \gamma^2 , \left(5 \beta \gamma^2\right)^\sigma \right)_2=1 \\
\implies 
    & \left(5 \beta , \left(5 \beta \right)^\sigma \right)_2=1 \quad \text{by a property of Hilbert symbols.}
\end{align*}

Using bimultiplicativity of the Hilbert symbol,
\[ 
\left(5 \beta , \left(5 \beta \right)^\sigma \right)_2= (5,5)_2 (\beta,5)_2 (5, \beta^\sigma)_2 (\beta, \beta^\sigma)_2.
\]

Notice that since 2 is inert in $K/\QQ$ and since 5 is invariant under the action of $\Gal(K/\QQ)$, applying the Galois action to the quadratic form for $(\beta,5)_2$ yields the form for $(5, \beta^\sigma)_2$ so the cross terms cancel one another. Therefore 
\[ 
\left(5 \beta , \left(5 \beta \right)^\sigma \right)_2= (5,5)_2 (\beta, \beta^\sigma)_2.
\]

Since $5 \times 2^2 + 5 \times 1^2 = 5^2$, $(5,5)_2 =1$. Therefore
\[
\left(5 \beta , \left(5 \beta \right)^\sigma \right)_2= (\beta, \beta^\sigma)_2
\]
so
\[
(\alpha,\alpha^\sigma)_2=1 \implies (\beta, \beta^\sigma)_2 =1.
\]
Therefore $\star$ is a well-defined map from $\mathbf{M}_4$. 

We now prove that if $\alpha,\beta\in\mathbf{M}_4$ are the in same Galois orbit, then $\star(\alpha)=\star(\beta)$. Let $\tau\in\Gal(K/\QQ)$ such that $\alpha^\tau=\beta$ for $\alpha,\beta\in\mathbf{M}_4$.

Suppose $(\alpha,\alpha^\sigma)_2=1 $ for all $\sigma\neq 1$ in $\Gal(K/\QQ)$. 
Then in $K_{(2)}$, the completion of $K$ at $2\OK$, there is a nontrivial solution $x,y,z$ to
\[
\alpha x^2 + \alpha^\sigma y^2 = z^2.
\]
Applying the action of $\tau$ yields a nontrivial solution to 
\[
\beta x^2 + \beta^\sigma y^2 = z^2
\]
so $(\beta,\beta^\sigma)_2=1$ for all $\sigma\neq 1$.
\end{proof}

Recall that by Lemma \ref{Mundy}, the elements of $\mathbf{M}_4$ that are invariant under the $\Gal(K/\QQ)$-action are exactly $\pm1$. The following lemma fully describes  $\star$ on these invariants. 

\begin{lem}\label{pm1}\switch{[pm1]}  Let $K:=K(n,\ell)$. 

\begin{enumerate}
    \item  $\star(1)=1$. 
    \item  $\star(-1)=-1$.
\end{enumerate}

\end{lem}

\begin{proof}
Observe that $(1,1)_2=1$ because $x^2 + y^2 = z^2$ has the solution $(x,y,z)=(1,0,1)$.


If $(-1,-1)_2=1$, there would be a non-trivial solution to $x^2+y^2+z^2\equiv 0 \bmod 4$. Since there is no such solution, $(-1,-1)_2=-1$.
\end{proof}




\begin{defn}\label{defn:MazurInvariant}\switch{[defn:MazurInvariant]}
Let $K:=K(n,\ell)$. Define the \textit{starlight invariant}, $m_K$ to be the number of $\Gal(K/\QQ)$-orbits $X$ of $\mathbf{M}_4$ of non-trivial size such that $\star(X)=1$. That is, for $\sigma$ a generator of $\Gal(K/\QQ)$,
\[
m_K := \#\{ X \in \mathbb{M}_{4,G} : \# X = n \text{ and } \star(X)=1 \}.
\]
\end{defn}

\begin{rmk}\label{rmk:MazurInvariant} By Lemma \ref{pm1}, it is equivalent to define the \textit{starlight invariant} of $K$, as
\[
m_K = \#\ker(\star)-1.
\]
 Here $\star$ refers to the map $\star: \mathbb{M}_{4,G} \to \pm 1$ given in Theorem \ref{StarMG}.
\end{rmk}

We now define
\[
\star: \PP_K^2 \to \{\pm1\} \quad \text{and} \quad
\star: \PP_\QQ^2 \to \{\pm1\}
\] 
to be the composition of $\star$ as defined in Theorem \ref{StarMG} with $\mathbf{r}_0$ and $\mathbf{r}$ respectively as defined in Definition \ref{defn:r}.

 \begin{defn}\label{defn:starp}\switch{[defn:starp]}

Let $p\in\PP_{\QQ}^{2}$ and let $\pp\in \PP_K^{2}$. 
 Define $\star(\pp) := \star\circ \mathbf{r}_0$ and $\star(p) := \star\circ \mathbf{r}$, the composition of the maps $\mathbf{r}_0$ and $\mathbf{r}$ respectively with the map $\star$ from Definition \ref{defn:r}.

We say that a prime $\pp\in\PP_K^{2}$ (respectively $p\in\PP_\QQ^{2}$) has property $\star$ or that $\star$ is true for $\pp$ (respectively $p$) whenever $\star(\pp)=1$ (respectively $\star(p)=1$). 
\end{defn}

Theorem \ref{thm:littledensity} and Theorem \ref{thm:littledensity:ext} give formulas in terms of $n$ and $m_K$ for the density of rational primes (assumed to split completely in Theorem \ref{thm:littledensity}) that satisfy property $\star$.

\section{Density Theorems}

 
 We first state and prove Theorem \ref{thm:littledensity} which gives 
a formula describing the {restricted density} of rational primes that satisfy the spin relation, the restriction being to primes that split completely in $K/\QQ$.
 Handling the inert case separately, we then apply Theorem \ref{thm:littledensity} to obtain Theorem \ref{thm:littledensity:ext} which gives a formula for the overall density of rational primes that satisfy the given spin relation. Lastly, we prove Theorems \ref{thm:Dstarbound} and \ref{thm:Dstarbound2} which give bounds on the densities given in Theorems \ref{thm:littledensity} and \ref{thm:littledensity:ext} respectively.
 
Recall the definitions of $S$, $S'$, $I$, and $I'$ from Definition \ref{defn:SI}. 

\begin{defn}\label{defn:RB}
Let $K:=K(n,\ell)$. Define the following sets of rational primes. 
\begin{align*}
B&:=\{ p\in \PP_\QQ^{2\ell}: \star(p) = 1  \} \\
R&:=B\cap S. 
\end{align*}
\end{defn}

Note that by Lemma \ref{spindep}, $B$ is exactly the set of rational primes in $p\in \PP_\QQ^{2\ell}$ such that
\[
\spin(\pp,\sigma) = \spin(\pp,\sigma\inv) \quad \text{for all } \sigma\in \Gal(K/\QQ) 
\]
where $\pp$ is a prime of $K$ above $p$. 

Recall from Definition \ref{defn:density} that  $d(R|S)$ denotes the  {restricted density} of primes $p\in R$ restricted to $S$. 

\begin{restatable}{thm}{littledensity}\label{thm:littledensity}
Let $K:=K(n,\ell)$. Then
\[
d(R|S) = \frac{1+ \text{$m_K$}n}{2^{n}}.
\]
\end{restatable}

\begin{proof} 

    
    Let $N\in \ZZ_{+}$. Let $R_N$ and $S_N$ denote the sets of  primes in $R$ and $S$ respectively of norm less than $N$. 
    We will show that 
    \begin{equation} \label{eqncombined}
    \lim_{N\to\infty} \frac{\#R_N}{\#S_N} = \frac{\#\{X\in \mathbf{M}_4: \star(X)=1\}}{\# \mathbf{M}_4} = \frac{1+m_Kn}{2^n}.
    \end{equation}
    
    Let $R'_N\subseteq \PP_K^{2\ell}$ denote the set of primes of $K$ that lay above rational primes in $R_N\subseteq \PP_\QQ^{2\ell}$ and define $S_N'$  similarly with respect to $S_N\subseteq \PP_\QQ^{2\ell}$. Let $\mathbf{r}_{0,N}$ denote the restriction of $\mathbf{r}_0$ to $S'_N \subseteq \PP_K^{2\ell}$.
    Since we have restricted to primes that split completely in $K/\QQ$, 
    \[
    \frac{\#R_N}{\#S_N} = \frac{\#R'_N}{\#S'_N}
    \]
    and  
    \[
    R'_N = \bigcup_{\star(\alpha)=1}  \mathbf{r}_{0,N}\inv(\alpha)
    \]
    where the above is a disjoint union over elements $\alpha\in \mathbf{M}_4$ such that $\star(\alpha)=1$. Therefore
    \[
    \frac{\#R'_N}{\#S'_N} = \frac{ 1}{\#S'_N} \sum_{\star(\alpha)=1} \#  \mathbf{r}_{0,N}\inv(\alpha).
    \]
    By Lemma \ref{equidistribution}, this implies
    \begin{align*}
    \lim_{N\to\infty} \frac{\#R'_N}{\#S'_N} &= \sum_{\star(\alpha)=1} \lim_{N\to\infty} \frac{ \#  \mathbf{r}_{0,N}\inv(X)}{\#S'_N}\\
    & =  \sum_{\star(\alpha)=1} \frac{1}{2^n}\\
    & = \frac{\#\{\alpha\in \mathbf{M}_4: \star(\alpha)=1 \}}{2^n}.
    \end{align*}
    This proves the first equality in equation \ref{eqncombined}. Let $\sigma$ be a generator of $\Gal(K/\QQ)$.
    By Lemma \ref{Mundy}, the elements of $\alpha\in \mathbf{M}_4$ such that $\alpha^\sigma = \alpha$ are $\alpha=\pm 1$ and we know that $\star(1)=1$ and $\star(-1)=-1$ by Lemma \ref{pm1}. Recalling that $m_K=\#\{[\alpha]\in \mathbb{M}_{4,G}: \alpha^\sigma \neq \alpha, \star(\alpha)=1 \}$, this implies 
    \[
    \#\{\alpha\in \mathbf{M}_4: \star(\alpha)=1\}= m_Kn+1.
    \]
    since $n=[K:\QQ]$ is prime so Galois orbits $X\in\mathbb{M}_{4,G}$ such that $X^\sigma \neq X$ each contain $n$ elements.
\end{proof}

We now state an extended version of Lemma \ref{equidistribution} which handles the inert case allowing us to give a formula in Theorem \ref{thm:littledensity:ext} for $d(B|\PP_\QQ^{2\ell})$, the overall density of rational primes that satisfy $\star$.

\begin{lem}\label{equidistribution:ext}\switch{[equidistribution:ext]}
Let $K:=K(n,\ell)$.

\begin{enumerate}
    \item\label{equidistribution:allp2}  For any $\alpha\in \mathbf{M}_4$, the density of $\pp\in \PP_K^{2\ell}$ such that $\varphi(\pp)=\alpha$ is $\frac{1}{2^n}$. That is, 
    \[
    d(\mathbf{r}_0\inv(\alpha)| \PP_K^{2\ell}) = \frac{1}{\#\mathbf{M}_4} = \frac{1}{2^n}.
    \]
    \item\label{equidistribution:spcomp2}  Restricting to primes of $K$ that split completely in $K/\QQ$,
    \[
    d(\mathbf{r}_0\inv(\alpha) \cap S' | S') = \frac{1}{\#\mathbf{M}_4} = \frac{1}{2^n}.
    \]
    \item\label{equidistribution:inert2} Restricting to inert primes of $K$, 
    \[
    d(\mathbf{r}_0\inv(\alpha) \cap I' | I') = \left\{
    \begin{array}{cc}
        \frac{1}{2} &  \text{if } \alpha = \pm1\\
        0 & \text{otherwise.}
    \end{array}
    \right.
    \]

\end{enumerate}
\end{lem}

\begin{proof}
Part (a) and part (b) were proven in Lemma \ref{equidistribution}. 

If $\alpha \neq \pm 1$ (for $\alpha\in \mathbf{M}_4$) then $\mathbf{r}_0\inv(\alpha) \cap I'=\emptyset$ since $\pm 1$ are the only invariants of the $\Gal(K/\QQ)$-action on $\mathbf{M}_4$ by Lemma \ref{Mundy}. 
Therefore $d(\mathbf{r}_0\inv(\alpha) \cap I' | I') = 0$ if $\alpha \neq \pm 1$.

Now fix $s=\pm 1$. Then
\[
\mathbf{r}_0\inv(s) \cap I' = 
        \left\{ \pp\in I' : \left( \frac{\alpha}{4} \right)_K = s  \right\}
\]
where $\left( \frac{\alpha}{4} \right)_K$ denotes the quadratic residue symbol in $\OK$ for $\alpha\in\OK$ a totally positive generator of $\pp^{h(K)}$. The quadratic residue condition is a congruence condition modulo $4\ZZ$ and being inert is a congruence condition with an odd modulus so the Chinese remainder theorem together with the cyclic case of Cebotarev's Density Theorem implies
\[
d(\mathbf{r}_0\inv(s) \cap I' | I') = \frac{1}{2}.
\]
 Theorem 4 in \cite{Serre} asserts Cebotarev's Density Theorem for natural density, or see \cite{NeukirchANT} Theorem VII.13.4 for an simpler proof using Dirichlet density.
\end{proof}


 We now prove the main results.

\littledensityext*

\begin{proof}

Recall Definition \ref{defn:RB}. Note that by Lemma \ref{spindep}, $B$ is the set of rational primes not dividing $2\ell$ that satisfy the given spin relation. Therefore
\[D_K=d(B|\PP_\QQ^{2\ell}).\]

Let $N\in \ZZ_{+}$. Let $I_N$ and $S_N$ denote the sets of (rational) primes in $I$ and $S$ respectively with positive generator less than $N$. 
Let $I'_N\subseteq \PP_K^{2\ell}$ denote the set of primes of $K$ which lay above rational primes in $I_N\subseteq \PP_\QQ^{2\ell}$ and define $S_N'$  similarly with respect to $S_N\subseteq \PP_\QQ^{2\ell}$. 
Note that while $S'_N= \{\pp\in S' : \Norm_{K/\QQ}(\pp)<N\}$,
\[
I'_N= \{\pp\in I' : \Norm_{K/\QQ}(\pp)<N^n\}.
\]

Since we have restricted to primes that are inert in $K/\QQ$, 
    \[
    \frac{\#B\cap I_N}{\#I_N} = \frac{\#B' \cap I'_N}{\#I'_N} 
    \]
where $B':=\{\pp\in \PP_K^{2\ell}: \star(\pp)=1\}=\{\pp\in \PP_K^{2\ell}: \pp \text{ lays above some } p\in B\}.$

Let $\mathbf{r}_{0,N}$ denote the restriction of $\mathbf{r}_0$ to $I'_N \subseteq \PP_K^{2\ell}$.
Observe that $\pp\in I'$ implies $\pp^\sigma=\pp$ so  $\mathbf{r}_0(\pp)=\pm 1$ for all $\pp\in I'$ by Lemma \ref{Mundy}. Lemma \ref{pm1} states that $\star(1)=1$ and $\star(-1)=-1$. Therefore
\[
B'\cap I'_N = \mathbf{r}_{0}\inv(1) \cap I_N'.
\]
Therefore $d(B'\cap I'|I') = \frac{1}{2}$ by part (c) of Lemma \ref{equidistribution:ext}. Then since $\frac{\#B\cap I_N}{\#I_N} = \frac{\#B' \cap I'_N}{\#I'_N} $, we have proven that
\begin{equation}\label{eqn:BI}
 d(B \cap I | I ) = \frac{1}{2}.
\end{equation}

Note that since $K/\QQ$ is cyclic,  $\PP_\QQ^{2\ell}$ is the disjoint union of $S$ and $I$.

Cebotarev's Density Theorem is true for natural density by Theorem 4 in \cite{Serre}. (Theorem VII.13.4 in \cite{NeukirchANT} gives a simpler proof using Dirichlet density.)

By Cebotarev's Density Theorem, $d(S|\PP_\QQ^{2\ell})= \frac{1}{n}$ and 
$d(I|\PP_\QQ^{2\ell})= \frac{n-1}{n}$. Therefore
\begin{align*}
    d(B | \PP_\QQ^{2\ell}) 
    & = \lim_{N\to\infty} \frac{\#B_N}{\#\PP_{\QQ,N}^{2\ell}} 
    = \lim_{N\to\infty} \left( \frac{\#B \cap I_N}{\#I_N}\frac{\#I_N}{\#\PP_{\QQ,N}^{2\ell}} +\frac{\#B \cap S_N}{\#S_N}\frac{\#S_N}{\#\PP_{\QQ,N}^{2\ell}}   \right) \quad \\
    & = \left( \frac{1}{2} \right)\left( \frac{n-1}{n} \right) + \left( \frac{m_Kn+1}{2^n} \right)\left( \frac{1}{n} \right) \quad \text{by Theorem \ref{thm:littledensity}}\\
    & = \frac{ 2^{n-1}(n-1)+ m_Kn+1}{2^nn}.
\end{align*}
\end{proof}



\begin{lem}\label{halfstar}\switch{[halfstar]}
Let $K:=K(n,\ell)$. For all $\alpha\in \mathbf{M}_4$, if $\star(\alpha)=1$ then $\star(-\alpha)=-1$.
\end{lem}

\begin{proof}
By Lemma \ref{pm1}, $(-1,-1)_2=-1$.

Next note that $(a,b)_2 = (a^\sigma,b^\sigma)_{2}$ for all $\sigma\in \Gal(K/\QQ)$ since $2$ is inert in $K$.

Assume $\star(\alpha)=1$. Then $(\alpha,\alpha^\sigma)_2=1$ for all nontrivial $\sigma\in\Gal(K/\QQ)$. Let $\sigma\in\Gal(K/\QQ)$ be nontrivial. By bimultiplicativity of Hilbert symbols,
\begin{align*}
(-\alpha,-\alpha^\sigma)_2 & = (-\alpha,-1)_2 (-\alpha,\alpha^\sigma)_2 \\
& = (-1,-1)_2 (\alpha,-1)_2 (-1,\alpha^\sigma)_2 (\alpha,\alpha^\sigma)_2.
\end{align*}

Next observe $(\alpha,-1)_2=(-1,\alpha)_2 =(-1,\alpha^\sigma)_2$, the second equality coming from the Galois-invariance shown earlier in this proof. Therefore $(\alpha,-1)_2 (-1,\alpha^\sigma)_2=1$. Then since $(\alpha,\alpha^\sigma)_2=1$ and $(-1,-1)_2=-1$, we get that
\[
(-\alpha,-\alpha^\sigma)_2=-1.
\]
Therefore $\star(-\alpha)=-1$.
\end{proof}

Recall the Definitions \ref{defn:SI} and  \ref{defn:RB} defining $S$ and $R$. 


\begin{restatable}{thm}{Dstarbound}
\label{thm:Dstarbound}
Let $K:=K(n,\ell)$. 
\[
 \frac{1}{2^n} \leq d(R|S) \leq \frac{1}{2}.
\]
\end{restatable}

\begin{proof}
By Theorem \ref{thm:littledensity},
\[
d(R|S)= \frac{1 + m_Kn}{2^n}=\frac{\#\{\alpha\in \mathbf{M}_4: \star(\alpha)=1\}}{2^n}.
\]

Lemma \ref{halfstar} implies the upper bound;  $\alpha \neq -\alpha$ in $\mathbf{M}_4$ because $-1$ is not a square modulo $4\OK$.

The lower bound is true because $\star(1)=1$ by Lemma \ref{pm1}
so 
\[
\#\{\alpha\in \mathbf{M}_4: \star(\alpha)=1\} \geq 1.
\]
\end{proof}

Theorem \ref{thm:Dstarbound2} is a Corollary of Theorem \ref{thm:Dstarbound} obtained from the fact that
\[
D_K = \frac{n-1}{2n} + \left(\frac{1}{n}\right)d(R|S)
\]
as in the proof of Theorem \ref{thm:littledensity:ext}.

\Dstarboundtwo*

\section*{Acknowledgements}

This research was conducted at Cornell University under the supervision of Ravi Ramakrishna. The paper was finalized at the Max Planck Institute for Mathematics.
In addition, the author would like to thank Brian Hwuang, Christian Maire, and Carlo Pagano for their helpful mathematical discussions and Pieter Moree for his guidance.



\bibliography{Thesisbib}
\bibliographystyle{alpha}

\end{document}